\documentclass[12pt,thmsa]{amsart}%

\usepackage[T1]{fontenc}
\usepackage[utf8]{inputenc}

\usepackage{amsmath}
\usepackage{amsthm}
\usepackage{graphicx}
\usepackage{amsfonts}
\usepackage{amssymb}

\providecommand{\U}[1]{\protect\rule{.1in}{.1in}}

\newtheorem{theorem}{Theorem}
\newtheorem{proposition}{Proposition}
\newtheorem{lemma}[proposition]{Lemma}

\theoremstyle{definition}
\newtheorem{definition}[proposition]{Definition}

\theoremstyle{remark}
\newtheorem{remark}{Remark}
\newtheorem{example}{Example}
\newcommand{\commentaire}[1]{}

\newcommand{\N}{\mathbb N}

\newcommand{\Q}{\mathbb Q}
\newcommand{\R}{\mathbb R}
\newcommand{\Z}{\mathbb Z}

\newcommand{\ttt}{\theta}
\newcommand{\LL}{\Lambda}
\newcommand{\fff}{\rightarrow}

\newcommand{\sgn}{\operatorname{sgn}}
\newcommand{\dd}{\operatorname{d}}
\newcommand{\card}{\operatorname{card}}
\newcommand{\eps}{\varepsilon}
\begin{document}
\title{The natural extension of the Gauss map and Hermite best approximations}
\maketitle

\begin{abstract}
	Hermite best approximation vectors of a real number $\ttt$ were introduced by Lagarias. A nonzero vector $(p,q)\in \Z\times\N$ is a Hermite best approximation vector of $\ttt$ if there exists $\Delta>0$ such that $(p-q\ttt)^2+q^2/\Delta\leq (a-b\ttt)^2+b^2/\Delta$ for all nonzero $(a,b)\in \Z^2$. Hermite observed that if $q>0$ then the fraction $p/q$ must be a convergent of the continued fraction expansion of $\ttt$ and Lagarias pointed out that some convergents are not  associated with a Hermite best approximation vectors. In this note we show that the almost sure proportion of Hermite best approximation vectors among convergents is $\ln 3/\ln 4$. The main tool of the proof is the natural extension of the Gauss map $x\in]0,1[\fff\{1/x\}$.
\end{abstract}

\section{Introduction}

In 1850, Hermite observed that  the  fractions $\tfrac{p_{\Delta}}{q_{\Delta}}$ associated with the minima $(p_{\Delta},q_{\Delta})\in\Z\times\N\setminus\{0\}$  of the quadratic forms
\[
f_{\Delta}(p,q)=(p-q\ttt)^2+\frac{q^2}{\Delta}, \,\Delta>0,
\] 
 are all convergents of the continued fraction expansion of the real number $\ttt$ (see \cite{Her}, p.295). However, some convergents are skipped. This was observed by Lagarias in \cite{Lag1} where he defined the Hermite best  approximations vectors of $\ttt$ as the pairs $(p,q)\in\Z\times\N$ that minimize  the quadratic forms $f_{\Delta}$ on $\Z^2\setminus\{(0,0)\}$ for at least one $\Delta>0$. The main objective of Lagarias' work was to define and to study the multidimensional Minkowski geodesic continued fraction algorithm. 
 In \cite{GrLa}, Grabiner and Lagarias studied the deep relationships between the one-dimensional Minkowski geodesic continued fraction algorithm, the additive and multiplicative continued  fraction algorithm, and the cutting sequences of the geodesic flow in the hyperbolic plane. The main objective of this note is to study the almost sure proportion of Hermite best approximation vectors among convergents of real numbers. 
 Here, we shall talk of best approximation vectors $(p_n,q_n)$ rather than convergents $\tfrac{p_n}{q_n}$. The only slight difference is that the convergent $\tfrac{p_0}{q_0}$ is skipped when the fractional part of $\ttt$ is $>\tfrac12$. We shall prove  
 
 \begin{theorem}
 	Let $\ttt$ be in $\R$ and let $(p_n,q_n)_{n\geq0}$ be the sequence of best approximation vectors  of $\ttt$.
 	\begin{enumerate}
 		\item For all $n\geq 0$, one at least of the best approximation vectors $(p_n,q_n)$ and $(p_{n+1},q_{n+1})$ is a Hermite best approximation.
 		\item For almost all $\ttt\in\R$,
 		\[
 		\lim_{n\fff\infty} \frac1n \card\{0\leq k<n: (p_k,q_k) \text{ is a Hermite best  approximation vector}\}=\frac{\ln 3}{\ln 4}.
 		\]
 		\item Let $(g_n,h_n)$ be the $n$-th Hermite best  approximation vector of $\ttt$. Then 
 		\[
 		\lim_{n\fff\infty}\frac1n \ln h_n=\frac{\pi^2}{6\ln 3}
 		\] 
 		for almost all $\ttt\in \R$.
 	\end{enumerate}
 \end{theorem}
  
The main ingredient of the proof of Theorem 1 is the natural extension of the Gauss map $x\in]0,1[\fff\{\tfrac1x\}$. The second objective of this note is to introduce the natural extension of the Gauss map starting from minimal vectors in lattices. There are many ways to introduce the natural extension of the Gauss map, see for instance see \cite{Na,ArNo,Sch} and,   
although the idea of minimal vectors  goes back to Voronoï (\cite{Vor}), it seems that their use for the natural extension of the Gauss map is not so well known, a use I learned from Yitwah Cheung. Recently, Yi Han, a student of Cheung, did a senior thesis where the same approach is explained with emphasis on the role of the diagonal flow, see \cite{Han}. 
 
The note is organize as follows, we first define minimal vectors in lattices of $\R^2$ and consecutive minimal vectors of these lattices. Then, we describe the algorithm that computes the minimal vector that immediately follows a pair of consecutive minimal vectors, this leads  to the definition of the natural extension of the Gauss map. Afterward, we state and prove all the results about the natural extension that are needed to prove Theorem 1, even those that  are well known. Among these intermediate results, Proposition \ref{prop:hermite} characterizes Hermite vectors with the natural extension.  Finally we prove Theorem 1. 

\section{Minimal vectors in lattices of $\R^2$}
{\bf Notation.} 
The box $B(a,b)$ is the set of vectors $(x,y)\in\R^2$ such that $|x|\leq a$ and $|y|\leq b$. 
When $u=(u_1,u_2)$ and $v=(v_1,v_2)$ are in $\R^2$, the box $B(u)$ is defined by $B(u)=B(|u_1|,|u_2|)$ and  the box $B(u,v)$ is defined by $B(u,v)=B(\max(|u_1|,|v_1|),\max(|u_2|,|v_2|))$.
\begin{definition}
	Let $\Lambda$ be a  lattice in $\mathbb{R}^{2}$.
	
	\begin{itemize}
		\item A nonzero vector $u=(u_{1},u_{2})\in\Lambda$ is a \textit{minimal}
		vector in $\Lambda$ if for every nonzero $v\in\Lambda$, $v\in B(u)\Rightarrow
		|v_{1}|=|u_{1}|$ and $|v_{2}|=|u_{2}|$.
		\item Two minimal vectors $u=(u_{1},u_{2})$ and $v=(v_{1},v_{2})$  are 
		\textit{consecutive} if $|u_2|<|v_2|$ and there are no minimal vector $w=(w_1,w_2)$ such that $|u_{2}|<|w_2|<|v_{2}|$.
		\item A sequence $X_n=(x_n,y_n)$, $n\in I$, is a \textit{complete} sequence of minimal vectors in $\LL$ if 
		 \begin{itemize}
		 	\item $I\subset \Z$ is an interval,
			\item  for all $n\in I$, $X_n$ is a minimal vector in $\LL$,
			\item for all $n\in I$ such that $n+1\in I$, $|y_{n+1}|>|y_n|$, 
			\item  for all minimal vectors $u=(x,y)$, there exists $n\in I$ such that $|y_n|=|y|$.
		\end{itemize}    
	\end{itemize}
\end{definition}

\begin{example}
	When 
	\[
	\LL_{\ttt}=\begin{pmatrix}
		1&-\ttt \\
		0&1
	\end{pmatrix}\Z^2=M_{\ttt}\Z^2.
	\]
	where $\ttt\in\R$, the vectors $(1,0)$ and $(-\ttt',1)$ with $\ttt'\in[-\frac12,\frac12]\cap(\ttt+\Z)$, are always consecutive minimal vectors in $\LL_{\ttt}$. 
\end{example}
\begin{remark}
	If $X_n$ and $X_{n+1}$ are two elements of a complete sequence of minimal vectors of a lattice, they are consecutive minimal vectors.
\end{remark}
\begin{remark}
	Since lattices are discrete subset, complete sequence of minimal vectors always exist.  These sequences are not unique, can be finite,  infinite one sided, or infinite two sided.
\end{remark}

Consider the lexicographic preorder on $\R^2$ defined by $(x_1,x_2)\prec(y_1,y_2) 
$
iff $|x_2|<|y_2|$ or $|x_2|=|y_2|$ and $|x_1|\leq |y_1|$.

\begin{lemma}\label{lem:consecutive}
	Two minimal vectors $u=(u_{1},u_{2})$ and $v=(v_{1},v_{2})$ in a lattice  $\LL\subset\R^2$ are
	consecutive iff $|u_{2}|<|v_{2}|$ and the only lattice point in the
	interior of $B(u,v)$ is zero.
\end{lemma}

\begin{proof}
	Let  $u=(u_{1},u_{2})$ and $v=(v_{1},v_{2})$ be two minimal vectors with $|u_2|<|v_2|$. If the set $\overset{o}{B}(u,v)\cap \LL\setminus\{0\} $ is nonempty, then it is finite and there is a $w$ minimal  for the lexicographic preorder $\prec$ in this set. On the one hand $w$ is a minimal vector in $\LL$. On the other hand, $|w_1|<|u_1|$ and $|w_2|<|v_2|$ and since $u$ is a minimal vector we have $|w_2|>|u_2|$. Hence $u$ and $v$ are not consecutive.
	
	Conversely, if $u$ and $v$ are not consecutive there is a minimal vector $w$ with $|u_2|<|w_2|<|v_2|$. Since $w$ is minimal $|u_1|>|w_1|$, hence $w\in  \overset{o}{B}(u,v)\cap \LL$. 
\end{proof}

\begin{proposition}\label{prop:minimalvector}
	Let $u=(u_1,u_2)$ and $v=(v_1,v_2)$ be two consecutive minimal vectors in a lattice $\LL\subset \R^2$. Then the pair
	$(u,v)$ is primitive, i.e. $\LL=\Z u+\Z v$.
\end{proposition}
\begin{proof}
	We can suppose $v_2>u_2\geq 0$. Since $v$ is minimal $|u_1|>|v_1|$. Let $w=(w_1,w_2)=xu+yv$ be in $\LL$ with $0\leq x,y<1$, we want to show that $x=y=0$. 
	
	Suppose $x+y\leq 1$. If $x,y>0$, then $w=(x+y)\tfrac{xu+yv}{x+y}$ is $(x+y)$ times a vector in the open line segment $]u,v[$, thus $w$ in the interior of the box $B(u,v)$ which contradicts Lemma \ref{lem:consecutive}. If $x=0$, then $w=yv$ and since $v$ is minimal, $y=0$.  If $y=0$, then $x=0$  as well.
	
	Suppose that $x+y>1$. The vector $w'=u+v-w=x'u+y'v$ is in $\LL$. Since $x'=1-x$ and $y'=1-y$ are both in $]0,1[$ and since $x'+y'=2-x-y<1$, $w'$ is in the interior of the box $B(u,v)$ which contradicts Lemma \ref{lem:consecutive}.
\end{proof}

\subsection{Minimal vectors and Diophantine approximations}
\begin{definition}
	Let $\ttt$ be a real number. A pair $(p,q)\in\Z\times\N^*$ is a best approximation vector of $\ttt$ if  for all $(a,b)\in\Z^2$,
	\begin{align*}
		\left\{
		\begin{array}[c]{l}
			0<|b|<|q|\Rightarrow |p-q\ttt|<|a-b\ttt|\\
			0<|b|\leq|q|\Rightarrow |p-q\ttt|\leq|a-b\ttt|
		\end{array}
		\right..
	\end{align*}
\end{definition}

\begin{proposition}\label{prop:best}
	Let $\ttt$ be a real number and consider  the lattice $\LL_{\ttt}$ defined by
	\[
	\LL_{\ttt}=\begin{pmatrix}
		1&-\ttt \\
		0&1
	\end{pmatrix}\Z^2=M_{\ttt}\Z^2.
	\]
	Then $X=\begin{pmatrix}
		x\\y
	\end{pmatrix}=M_{\ttt}\begin{pmatrix}
		p\\q
	\end{pmatrix}\in\LL_{\ttt}
	$ is a minimal vector with $y\neq 0$ iff $(p,q)$ is a best  approximation vector of $\ttt$.
\end{proposition}

\begin{proof}
	Suppose that $X=\begin{pmatrix}
		x\\y
	\end{pmatrix}$ is a minimal vector with $y\neq 0$. If $a$ and $b$ are  integers with $0<|b|<|y=q|$, then $Y=\begin{pmatrix}
		a-b\ttt \\
		b
	\end{pmatrix}\notin B(X)$ which implies $|a-b\ttt|>|p-q\ttt|$. If $|b|=|q|$ and if $Y\in B(X)$ then $|a-b\theta|=|p-q\theta|$.
	
	Conversely, if $(p,q)$ is a best  approximation vector of $\ttt$, then for any $(a,b)\in\Z^2$, $Y=\begin{pmatrix}
		a-b\ttt \\
		b
	\end{pmatrix}\in B(X)$ implies
	\[\left\{\begin{array}{ll}
		|a-b\ttt|\leq |p-q\ttt|\\ |b|\leq|q|
	\end{array}\right..
	\]
	If $b\neq 0$, this in turn implies $|a-b\ttt|=|p-q\ttt|$ and $|b|=|q|$ by definition of best approximation vectors. If $b=0$ and $a\neq 0$ then $|a|\geq 1>\tfrac{1}{2}\geq |p-q\ttt|$, hence $Y\notin B(X)$.
\end{proof}

\section{Minimal vectors and the natural extension of the Gauss map}
\subsection{Definition of the natural extension}

Let denote $\lfloor x\rfloor $ the lower integer part of the real number $x$ and $\{x\}=x-\lfloor x\rfloor $ its fractional part. Set
\[
U=]0,1[^2\,\cup \,(\{0\}\times [0,\tfrac12])\cup ([0,\tfrac12]\times\{0\}).
\]
\begin{proposition} \label{prop:gauss} Let $\LL$ be a lattice in $\R^2$ and let $u=(u_1,u_2)$ and $v=(v_1,v_2)$ be a pair of consecutive minimal vectors in $\LL$ with $u_2,v_2\geq 0$ and $v_2>u_2$. 
	\begin{enumerate}
	\item Then	
	  $u_1\neq 0$ and there exist $(x,y)\in U$ and $\eps\in\{-1,1\}$ unique such that
		\begin{align*}
			u&=(u_1,u_2)=(\eps |u_1|,v_2y)\\
			v&=(v_1,v_2)=(-\eps |u_1|x,v_2).
		\end{align*}
	(In the case $u_2=0$ and $v_1=\tfrac12 u_1$ we change $u$ in $-u$.)
	\item If $v_1\neq 0$, then, with $a=\lfloor \tfrac1x\rfloor $, 
	\[
	w=u+av
	\]
	 is the minimal vector that follows immediately $v$. Furthermore, 
	\begin{align*}
		v&=(\eps'|v_1|,w_2y')\\
		w&=(-\eps'|v_1|x',w_2)
	\end{align*}
	where
	\begin{align*}
		&\eps'=-\eps,\,w_2=v_2(a+y),\\
		 &x'=\{\tfrac1{x}\} \text{ and } y'=\frac{1}{a+y}.
	\end{align*}
	\end{enumerate}
\end{proposition}
\begin{proof}
	1. Let $u=(u_1,u_2)$ and $v=(v_1,v_2)$ be a pair of consecutive minimal vectors with $u_2$ and $v_2$ non negative and $u_2<v_2$. Since $v$ is minimal, we have $|u_1|>|v_1|$, $u_1=\eps|u_1|$, $v_1=\alpha \eps |u_1|x$ and $u_2=v_2y$ where $\eps,\alpha=\pm 1$ and $x,y\in[0,1[$. Consider the vectors
	\begin{align*}
	&w= u-\alpha v=(\eps|u_1|(1-x),v_2(y-\alpha)),\\
	&w'=u+\alpha v=(\eps|u_1|(1+x),v_2(y+\alpha)).
	\end{align*}
\begin{enumerate}
	\item If $x,y>0$ and  $\alpha=1$ then $w$ is in the interior of the box $B(u,v)$, which contradicts Proposition \ref{lem:consecutive}. \\
	\item If $x=0$, then  $|y-\alpha|$ and $|y+\alpha|$ are $\geq y$ because $u$ is minimal. It implies $y\leq 1/2$. \\
	\item If $y=0$, then  $|1-x|$ and $|1+x|$ are $\geq x$ because $v$ is minimal. It implies $x\leq 1/2$.
	\end{enumerate}
	
	It follows that $(x,y)\in U$. In  case (1), when $x,y>0$ only $\alpha=-1$ is possible. In  case (2), $\alpha$ can be either $1$ or $-1$ and  case (3), we can suppose $\alpha=-1$ by changing $u$ in $-u$ if necessary. In the three cases $\alpha=-1$ works. Finally, $\eps,x$ and $y$ are unique because $|u_1|$ and $v_2$ are $>0$.
	
	2. If $v_1\neq 0$, then $x>0$ and the  vector $u+av$ is in  the strip $\{(x_1,x_2):|x_1|<|v_1|\}$ for is first coordinate is $\eps|u_1|x(\tfrac1x-a)=-v_1\{\tfrac1x\}$. If $w=(w_1,w_2)\in\LL\setminus\{0\} $ is minimal for the lexicographic preorder $\prec$ in this strip then $w$  is the minimal vector that immediately follows $v$.
	By Proposition \ref{prop:minimalvector}, $\LL=\Z u+\Z v=\Z v+\Z w$, hence $\det_{(u,v)}(v,w)=\pm 1$. Therefore $w=\pm u +nv$ where $n\in \Z$. We can suppose $w_2>0$. It implies $n>0$ for $v_2>u_2\geq 0$. Now, 
	\[
	|w_1|=|nv_1\pm u_1|=|v_1||\pm \tfrac1x -n|<|v_1|,
	\]
	hence $|\pm \tfrac1x -n|<1$. Since $n\geq 1$, the sign $\pm$ must be $+$ and $n=\lfloor \tfrac1x\rfloor =a$ or $a+1$. Since $w_2=u_2+nv_2$ and since $w$ is minimal for the lexicographic preorder $\prec$, $n=a$. Therefore 
	\[
	w=u+av
	\]
	 Finally, we obtain
	\begin{align*}
		v&=(\eps'|v_1|,w_2y')\\
		w&=(-\eps'|v_1|x',w_2)
	\end{align*}
	where
	\begin{align*}
		\eps'=-\eps,\, x'=\{\tfrac1{x}\} &\text{ and } y'=\frac{1}{a+y}.
	\end{align*}	
\end{proof}	
\begin{definition}\label{def:intrinsic}
	Let $\LL$ be a lattice in $\R^2$.
	\begin{enumerate}
		\item Let $u=(u_1,u_2)$ and $v=(v_1,v_2)$ be two consecutive minimal vectors in $\LL$. The triple $(\eps,x,y)\in\{-1,1\}\times U$ associated with $(u,v)$ by Proposition \ref{prop:gauss} is called the {\sl intrinsic} coordinates of the pair $(u,v)$.
		\item The map $T:]0,1[^2\cup (]0,\tfrac12]\times \{0\})\fff [0,1[^2$ defined by
		\[
		T(x,y)=(\{\tfrac1x\},\frac1{\lfloor \tfrac1x\rfloor +y})
		\] is  ``the natural extension'' of the Gauss map.
	\end{enumerate}
\end{definition}
\begin{remark}
	Natural extensions of measure preserving maps were introduced  by Rohlin in 1961 (see \cite{Ro}). Here we shall not prove that the map $T$ is the natural extension of the Gauss map, i.e., is the ``smallest'' invertible extension of the Gauss map.  We shall only prove that it is invertible and measure preserving.
\end{remark}
\subsection{Properties of the natural extension}
\begin{lemma}
	$T$ is one to one and $T(]0,1[^2\cup (]0,\tfrac12]\times \{0\}))=]0,1[^2\cup (\{0\}\times ]0,\tfrac12])$. Furthermore,
	for $(x',y')\in U\setminus(]0,\tfrac12])\times\{0\}$, 
	\[
	T^{-1}(x',y')=(\tfrac1{\lfloor \tfrac1{y'}\rfloor +x'},\{\tfrac1{y'}\}).
	\]
\end{lemma}	
\begin{proof}
If $T(x,y)=(x',y')$ then $0< y'=\frac1{\lfloor \tfrac1x\rfloor +y}<1$. With $b=\lfloor \tfrac1{y'}\rfloor$, we have  , $ b\leq \lfloor\tfrac1x\rfloor +y<b+1$, which implies $\lfloor \tfrac1x \rfloor=b$.  In turn this implies $x'=\tfrac1x-b$ and $y=\tfrac1{y'}-b$ and then   $(x,y)=(\tfrac1{b+x'},\{\tfrac1{y'}\})$. Therefore $T$ is one to one. Moreover, if $x'=0$, then $x=\tfrac1b$ and $b$ cannot be $1$ so that $y'\leq \tfrac12$. It follows that $T(]0,1[^2\cup (]0,\tfrac12]\times \{0\}))\subset ]0,1[^2\cup (\{0\}\times ]0,\tfrac12])$ . Conversely, it is easy to check that if $(x',y')\in]0,1[^2\cup (\{0\}\times ]0,\tfrac12])$ then $T(\tfrac1{\lfloor \tfrac1{y'}\rfloor +x'},\{\tfrac1{y'}\})=(x',y')$.	
\end{proof}
\begin{lemma}[Contraction Lemma]
	Let $x\in]0,1[$ be such that $T^2(x,0)$ is defined. Then for any $y,z\in]0,1[$, the four pairs $(x',y')=T(x,y)$, $(x'',y'')=T^2(x,y)$, $(x',z')=T(x,z)$ and $(x'',y'')=T^2(x,z)$ are defined and 
	\[
	|z'-y'|\leq |z-y| \text{ and } |z''-y''|\leq \tfrac12 |z-y|.
	\]
\end{lemma}
\begin{proof} With $a=\lfloor \tfrac1x\rfloor $ and $a'=\lfloor \tfrac1{x'}\rfloor $, we have
	\begin{align*}
		&y'=\frac1{a+y},\,z'=\frac1{a+z},\\
		&|z'-y'|=\frac{|y-z|}{|a+z||a+y|}\leq |z-y|,\\
		&y''=\frac1{a'+\frac1{a+y}}=\frac{a+y}{1+aa'+a'y},\,z''=\frac{a+z}{1+aa'+a'z},\\
		&z''-y''= \frac{(1+aa'+a'y)(a+z)-(1+aa'+a'z)(a+y)}{(1+aa'+a'z)(1+aa'+a'y)},\\
		&|z''-y''|\leq \frac{|z-y|}{(1+aa')^2}\leq \tfrac12|z-y|.
	\end{align*} 
\end{proof}
\begin{lemma}
	The probability $\mu=\frac{1}{\ln2(1+xy)^2}dxdy$ on $U$ is $T$-invariant and ergodic. 
\end{lemma}
\begin{proof}
	Since $T$ is a diffeomorphism from $(]0,1[\setminus\{\tfrac1n:n\in\N^2\})\times ]0,1[$ onto $]0,1[\times(]0,1[\setminus\{\tfrac1n:n\in\N^2\})$, it suffices to check that $f\circ T\times |\operatorname{Jac}T|=f$ where $f$ is the density of $\mu$. This verification is straightforward.\medskip
	
	By the contraction Lemma, for all $x\in ]0,1[\setminus \Q$ and all $y,z\in]0,1[$, 
	\[
	\lim_{n\fff\infty}\dd(T^n(x,y),T^n(x,z))=0.
	\]
	 So that, if $f:U\fff\R$ is continuous and if for some $(x,y)\in]0,1[^2$,
	$\lim_{n\rightarrow +\infty }\frac{1}{n
	}\sum_{k=0}^{n-1}f\circ T^{k}(x,y)=l(x,y)$, 
	then for all $z\in]0,1[$, 
	$\lim_{n\rightarrow +\infty }\frac{1}{n}\sum_{k=0}^{n-1}f\circ
	T^{k}(x,z)=l(x,y)$. Therefore, by Birkhoff Theorem, for almost all $(x,y)\in U$, the sequence $\frac{1%
	}{n}\sum_{k=0}^{n-1}f\circ T^{k}(x,y)$ converges to a limit $l(x)$ which does not depend on $y$. Since $ T^{-1}=s\circ T\circ s$ where $s(u,v)=(v,u)$, we also have that for almost all
	$(x,y)$, the sequence 
	\[
	\frac{1}{n}\sum_{k=0}^{n-1}f\circ T^{-k}(x,y)=\frac{1}{n}\sum_{k=0}^{n-1}f\circ s\circ T^{-k}(y,x)
	\]
	converges to a limit $l^{\prime }(y)$ which does not depend on  $x$. Since the forward limit and the backward limit are almost surely equal, it follows that $l(x)=l^{\prime }(y)$ for almost all $(x,y)$. Therefore, the sequence 
	$\frac{1}{n}\sum_{k=0}^{n-1}f\circ T^{k}(x,y)$ converges almost everywhere to a constant that must be the
	mean $M(f)=\int_Uf\,d\mu$. By  Lebesgue Theorem, the convergence also holds  in $L^1(\mu)$. It follows that the sequence of linear maps $A_{n}f=\frac{1}{n}%
	\sum_{k=0}^{n-1}f\circ T^{k}$ converges in $L^{1}(\mu )$ on an everywhere dense set of continuous functions to $M(f)$. Since the sequence of linear maps $(A_n)_n$ is bounded for the operator norm in $L^1(\mu)$, it follows that for all $f\in L^{1}$, $A_{n}f\fff M(f)$
	in $L^{1}(\mu)$ which implies that $T$ is ergodic.
\end{proof}

\section{Hermite best approximations vectors}
Recall that a shortest vector in a lattice with respect to a norm $\|.\|$ is a nonzero vector of the lattice whose norm is minimal.
\begin{definition}	
A {\sl Hermite  vector}  in a lattice $\LL\subset\R^2$ is a vector $w$ in $\LL$ that  is a shortest vector in  $\LL$ for an Euclidean norm $|(x_1,x_2)|_t^2= |tx_1|^2+|\tfrac1tx_2|^2$ where $t$ is a positive real number.\\
A {\sl Hermite best approximation vector} of $\ttt\in \R$ is a pair $(p,q)\in \Z\times\N$ such that $(p-q\ttt,q)$ is a Hermite vector in $\LL_{\ttt}$.
\end{definition}
 \begin{lemma}\label{lem:HermiteMinimal}
 	If $u=(u_1,u_2)$ is a Hermite vector in a lattice $\LL\subset\R^2$, then  $u$ is a minimal vector in $\LL$.
 \end{lemma}
\begin{proof}
	By definition of Hermite vector, there exists $t>0$ such that $u=(u_1,u_2)$ is a shortest vector for the norm $|.|_t$. Since \[
	B(u)\subset \{v\in\R^2:|v|_t\leq|u|_t\}
	\]
	 and
	\[
    B(u)\setminus \{v\in\R^2:|v|_t<|u|_t\}=\{(\pm u_1,\pm u_2)\},
 \]  
 $u$ is a minimal vector in $\LL$.
\end{proof}

\begin{lemma}\label{lem:hermite}
	Let $u=(u_1,u_2)$ be a  Hermite vector of a lattice $\LL\subset\R^2$. If with $|u_1|>0$, then
	\begin{enumerate} 
		\item there exists a Hermite vector $h=(h_1,h_2)$ with $|h_1|<|u_1|$,
		\item if $h$ is such a Hermite vector with $|h_2|$ minimal, then there exists a positive real number $t$ such that $u$ and $h$ are shortest vectors of $\LL$ with respect to the same norm $|.|_t$.
	\end{enumerate}
\end{lemma}

\begin{proof}
	1. Since $|u_1|>0$, there exists at least one nonzero vector $v=(v_1,v_2)\in\LL$ such that $|v_1|<|u_1|$. For $s>0$, large enough, $|v|_s^2=s^2|v_1|^2+\tfrac1{s^2}|v_2|^2<s^2|u_1|^2$. Let $h=(h_1,h_2)$ be a shortest vector in $\LL$ for the norm $|.|_s$. Then $s^2|h_1|^2\leq |h|_s^2\leq |v|_s^2<s^2|u_1|^2$.
	
	2. Suppose now $h=(h_1,h_2)$ is a hermite best approximation vector  with $|h_1|<|u_1|$ with $|h_2|$  minimal. . Let
	Let
	 \[
	t=\sup\{s>0:u \text{ is a shortest vector with respect to the norm } |.|_s\}.
	\]
	By continuity, we see that $u$ is still a shortest vector with respect to the norm $|.|_t$.
	We want to show that $|u|_t=|h|_t$. We use the following short steps:
	\begin{itemize}		
		\item If $v=(v_1,v_2)$ and $w=(w_1,w_2)$ are two Hermite vectors and if $|w_1|<|v_1|$ then $|w_2|>|v_2|$ because $w$ is a minimal vector. Therefore the function $s\fff |v|_s^2-|w|_s^2$ is strictly increasing. 
		\item There exists $r\geq t$ such that $h$ is a shortest vector with respect to the norm $|.|_r$. Otherwise, there exists $r<t$ such that $h$ is a shortest vector with respect to $|.|_r$ and we would have 
		$|u|_t-|h|_t>|u|_r-|h|_r\geq 0$ and $u$ would not be a shortest vector with respect to $|.|_t$.
		\item If $v=(v_1,v_2)$ is a shortest vector of $\LL$ with respect to $|.|_s$ for some $s>t$ then $|v_1|\leq |u_1|$. Otherwise, $|v|_t-|u|_t<|v|_s-|u|_s\leq 0$ and $u$ would not be a shortest vector with respect to $|.|_t$. 
		\item Since $\LL$ is discrete there exists a vector $v=(v_1,v_2)$ and a sequence $(s_n)$ of real numbers (strictly) decreasing to $t$ such that $v$ is a shortest vector with respect to the norm $|.|_{s_n}$ for each $n$.
		\item We have $|v_1|\neq|u_1|$. Otherwise $|v_2|=|u_2|$ and $u$ would be a shortest vector with respect to a norm $|.|_{s_n}$ with $s_n>t$.
		\item We have $|v_1|<|u_1|$, otherwise 
		$|v|_t-|u|_t<|v|_{s_n}-|u|_{s_n}<0$.
		\item If $|h_1|=|v_1|$ we are done.
		\item If $|h_1|\neq |v_1|$ then by definition of $h$ we have $|h_2|<|v_2|$ and therefore $|v_1|<|h_1|$. It follows that
		$|h|_t-|v|_t\leq |h|_r-|v|_r\leq 0$ and therefore $|h|_t\leq|v|_t=|u|_t$.
	\end{itemize}
\end{proof}

\begin{proposition}\label{prop:hermite} Let $\LL$ be a lattice in $\R^2$ and let $u=(u_1,u_2)$ and $v=(v_1,v_2)$ be a pair of consecutive minimal vectors in $\LL$ with $u_2,v_2\geq 0$ and $v_2>u_2$. Let $(\eps,x,y)\in\{-1,1\}\times U$ be the intrinsic coordinates of $(u,v)$ (see definition \ref{def:intrinsic}).
\begin{enumerate}
\item[a.] One at least of the two vectors $u$ and $v$ is a Hermite vector.
\item[b.] $u$ is a Hermite vector and $v$ is not a Hermite vector iff
\[
x>\frac{2y+1}{y+2}.
\]
Furthermore, if this inequality holds then $v$ and $v+u$ are consecutive minimal vectors.
\end{enumerate}
\end{proposition}

\begin{proof} 

1. Let us show that if $u$ is a Hermite vector  and $v$ is not a Hermite vector then $w=u+v$ is a Hermite vector and is the minimal vector that follows $v$. We proceed by contradiction and suppose that $w$ is not a Hermite vector.

Call $h=(h_1,h_2)$ the Hermite vector with $|h_1|<|u_1|$ and $h_2$ non negative and minimal. 
By Lemma \ref{lem:hermite}, there exists a $t>0$ such that $u$ and $h$ are shortest vectors of $\LL$ for the same norm $|.|_t$. Since $v$ is not a Hermite vector, $v$ is not a shortest vector of $\LL$ for the norm $|.|_t$, hence
\[
t^2|u|_t^2=t^4u_1^2+u_2^2=t^2|h|_t^2=t^4h_1^2+h_2^2<t^2|v|_t^2=t^4v_1^2+v_2^2.
\]
It follows that $t^4=\frac{h_2^2-u_2^2}{u_1^2-h_1^2}$ and that $t^4(u_1^2-v_1^2)+u_2^2-v_2^2<0$, hence
\[
\Delta=(h_2^2-u_2^2)(u_1^2-v_1^2)+(u_1^2-h_1^2)(u_2^2-v_2^2)<0.
\]

Let us show that  $|v_1|\leq \tfrac12 |u_1|$ is not possible.
Set $a=\lfloor \tfrac{|u_1|}{|v_1|}\rfloor$. On the one hand, by Proposition \ref{prop:gauss}, $v$ and  $u+av$ are the two minimal vectors that follows $u$, on the other hand, $h$ is a minimal vector that follows $u$ and $h\neq v$, hence $h=u+av$ or $h$ is after $u+av$. In both cases,   $h_2\geq u_2+av_2$. Since $a\geq 2$, we have 
\begin{align*}
\Delta&\geq (4v_2^2+4v_2u_2)\tfrac34 u_1^2+u_1^2(u_2^2-v_2^2)\\
&=2u_1^2v_2^2+3u_1^2u_2v_2+u_1^2u_2^2>0,
\end{align*}
a contradiction.
It follows that $|v_1|> \tfrac12 |u_1|$ and $a=1$. 

Since $a=1$, $w=u+v=(w_1,w_2)$ is a minimal vector and $h\neq w$ because we have assumed that $w$ is not a Hermite vector. 
Now,  $w=(w_1,w_2)$ where $|w_1|=|u_1|-|v_1|$ and $w_2=u_2+v_2$, hence  
\[
h_2\geq w_2+v_2=2v_2+u_2. 
\]
We have $|w|_t>|u|_t$, hence
\begin{align*}
\Delta'&=t^2(|u|_t^2-|w|_t^2)<0. 
\end{align*}
But
\begin{align*}
(u_1^2-h_1^2)\Delta'&=(h_2^2-u_2^2)(u_1^2-w_1^2)+(u_1^2-h_1^2)(u_2^2-w_2^2)\\ &\geq(4v_2^2+4v_2u_2)(2|u_1||v_1|-v_1^2)+u_1^2(-2u_2v_2-v_2^2)\\
&\geq(4v_2^2+4v_2u_2)|u_1||v_1|+u_1^2(-2u_2v_2-v_2^2)\\
&\geq(4v_2^2+4v_2u_2)\tfrac12 u_1^2+u_1^2(-2u_2v_2-v_2^2)\\
&\geq 0
\end{align*}
a contradiction. Finally, by Proposition \ref{prop:gauss}, $w$ is the minimal vector that follows $v$.

2. For every minimal vector $u=(u_1,u_2)$ with $u_2>0$, there exists a Hermite vector $w=(w_1,w_2)$ with $0\leq w_2<u_2$. Indeed, if $\LL$ contains a nonzero vector whose second coordinate vanishes, just take $w=(w_1,0)$ with $0\leq w_1$ minimal. $w$ is a Hermite vector with respect to $|.|_t$ when $t>0$ is small enough. Otherwise there is a vector $x=(x_1,x_2)\in\LL$ with $0<x_2<u_2$. One can find $t>0$ such that 
\[
\frac{1}{t^2}|x_2|^2>|u|_t^2.
\]
A shortest vector $w=(w_1,w_2)$ associated with such a $t$, is by definition a Hermite vector and we have $0<|w_2|<u_2$ because $\tfrac1{t^2}|w_2|\leq |w|_t^2\leq |u|_t^2<\frac{1}{t^2}|x_2|^2$.
 
3. (b) follows immediately from 1 and 2. 

4. Let $u$ and $v$ be two consecutive minimal vectors. By (a), with $r=|u_1|,q=v_2>0$ there exist $0\leq x,y<1$ such that $u=(\eps r,qy)$ and $v=(-\eps rx,q)$. Suppose that $u$ is a Hermite vector and $v$ is not. Let $w=u+v$. By 1, there exists $t>0$ such that 
\[
|u|_t=|w|_t<|v|_t.
\]
As in 1, this implies
\begin{align*}
&t^4=\frac{q^2((1+y)^2-y^2)}{r^2(1-(1-x)^2)},\\
&t^4r^2(1-x^2)+q^2(y^2-1)<0,
\end{align*}
thus
\[
\frac{((1+y)^2-y^2)}{(1-(1-x)^2)}(1-x^2)+(y^2-1)<0.
\]
which is equivalent to
\[
(2y+y^2)x^2+2(1-y^2)x-(2y+1)>0
\]
Solving in $x$, the discriminant is $(1-y^2)^2+(2y+y^2)(2y+1)=(1+y+y^2)^2$, thus we obtain
\[
x>\frac{2y+1}{y+2} \text{ or } x<-\frac1y
\]
and since $x\geq 0$, $x>\frac{2y+1}{y+2}$.

5. Conversely if the inequality $x>\frac{2y+1}{y+2}$ holds then with the value $t^4=\tfrac{q^2((1+y)^2-y^2)}{r^2(1-(1-x)^2)}$, we obtain
\[
|u|_t=|w=u+v|_t<|v|_t
\]
which implies that $v$ is not a Hermite vector. Actually,
if $s>t$ then $|w|_s<|v|_s$ and if $s<t$ then $|u|_s<|v|_s$.

6. It remains to show that if 
\[
x>\frac{2y+1}{y+2},
\]
then $u$ is a Hermite vector. Suppose on the contrary that $u$ is not a Hermite vector. If $v_1=0$, $v$ is a shortest vector of $\LL$ with respect to $|.|_s$ when $s$ is large enough, so that $v_1\neq 0$. Inverting the role of the first and of second coordinate, and using the steps 1 and 2, we see that $u$ or $v$ is a Hermite vector and therefore $u$ is a Hermite vector.
\end{proof}

\section{Proportion and growth rate of Hermite best approximations}

\begin{lemma}
	Let $V=\{(x,y)\in U: x>\tfrac{2y+1}{y+2} \}$. Then
	\[
	\int\int_V\frac{1}{(1+xy)^2}dxdy=\ln 2-\tfrac12\ln3
	\]
\end{lemma}	

\begin{proof}
	The lemme follows from the two standard calculations:
	\begin{align*}
		\int_{\tfrac{2y+1}{y+2}}^1\frac{1}{(1+xy)^2}dx&=\left[-\frac1y\frac{1}{1+xy}\right]_{\tfrac{2y+1}{y+2}}^1\\
		&=\frac{1-y}{2(1+y)(1+y+y^2)	}
	\end{align*}
and
\begin{align*}
	\int\frac{1-y}{2(1+y)(1+y+y^2)}dy=-2\ln(1+y)+\ln(1+y+y^2).
\end{align*}
\end{proof}
\begin{lemma}
	Let $\LL$ be  a lattice in $\R^2$ and let $X_n=(r_n,q_n)$, $n\in I\subset \Z$, be a complete sequence of minimal vectors  with $q_n\geq 0$ for all $n$. Suppose that $0,1\in I$ and let $(\eps,x,y)$ be the intrinsic coordinates of the pair $(X_0,X_1)$. 
	Then, for all $n\in I$ such that $n+1\in I$, $X_{n+1}$ is not a Hermite vector of $\LL$ iff $T^n(x,y)\in V$.
\end{lemma}
\begin{proof}
	By Proposition \ref{prop:gauss}, for all $n\in I$ such that $n+1\in I$, the intrinsic coordinate of $(X_n,X_{n+1})$ are
	\begin{align*}
		\eps_n=(-1)^n\eps,\, (x_n,y_n)=T^{n}(x,y).
	\end{align*}
By Proposition \ref{prop:hermite}, $X_{n+1}$ is not a Hermite vector iff $T^n(x,y)=(x_n,y_n)\in V$.	
\end{proof}	
	
\begin{lemma}
	For almost all $\ttt\in\R$,
	\[
	\lim_{n\fff\infty} \frac1n \card\{0\leq k<n: X_{k+1}(\ttt) \text{ is a Hermite vector}\}=\frac{\ln 3}{2\ln 2}.
	\]
	where $(X_n(\ttt))_{n\geq 0}$ is the complete sequence of minimal vectors of the lattice $\LL_{\ttt}$.
\end{lemma}	
\begin{proof}
	Let $\ttt$ be in $\R$. The first two minimal vectors of $\LL_{\ttt}$ are $X_0=(1,0)$ and $X_1=(-\ttt',1)$ where $\ttt'=\ttt-[\ttt]$ and $[\ttt]$ is the integer nearest to $\ttt$. The intrinsic coordinates of these two consecutive minimal vectors   are $(\eps,x,0)=(\sgn\ttt',|\ttt'|,0)$.
	So that, thanks to the previous Lemma,  it is enough to prove that
	\[
	\lim_{n\fff\infty}\frac1n\sum_{k=0}^n 1_V\circ T^k(x,0)=1-\frac{\ln 3}{2\ln 2}
	\]
	for almost all $x\in[0,\tfrac12]$. 
	By Birkhoff Theorem applied to the natural extension of the Gauss map and to the indicator function $1_V$, we know that for almost all $(x,y)\in U$,
	\[
	\lim_{n\fff\infty}\frac1n\sum_{k=0}^n 1_V\circ T^k(x,y)=1-\frac{\ln 3}{2\ln 2}.
	\]
	The problem is that the limit hold for almost all $(x,y)$ and not for almost all $x$.
	
	Suppose on the contrary that there exist $a>0$ and a measurable set $S\subset[0,\tfrac12]$ of positive measure such that for all $x\in S$
	\[
	\limsup_{n\fff\infty}\frac1n\sum_{k=0}^n 1_V(T^k(x,0))\geq1-\frac{\ln 3}{2\ln 2}+a
	\]
	or 
	\[
	\liminf_{n\fff\infty}\frac1n\sum_{k=0}^n 1_V(T^k(x,0))\leq1-\frac{\ln 3}{2\ln 2}-a.
	\]
	We deal with the first case, the second is similar. Let $t$ be positive real number and let 
	\[
	V_t=\{(x,y)\in U:\exists (y',x)\in U, |y-y'|\leq t\}.
	\]  
	We can choose $t$ small enough so that $\mu(V_t)<\mu(V)+\tfrac{a}{2}$.
	By the contraction Lemma, for all $(x,y)\in S\times [0,t]$ and all integers $n\geq 0$, 
	\[
	T^n(x,0)-T^n(x,y)=(0,z_n)
	\]
	with $|z_n|\leq t$.
	It follows that for all $n$ and  all $(x,y)\in S\times [0,t]$,
	\begin{align*}
	\limsup_{n\fff\infty}\frac1n\sum_{k=0}^n 1_{V_t}(T^k(x,y))&\geq \limsup_{n\fff\infty}\frac1n\sum_{k=0}^n 1_V(T^k(x,0))\\
	&\geq \mu(V)+a\\
	&\geq \mu(V_t)+\frac{a}{2},
	\end{align*} 
	Which contradicts Birkhoff Theorem used with the function $1_{V_t}$.
\end{proof}
\begin{proof}[End of proof of Theorem 1]
	(1) in Theorem 1 is the particular case $\LL=\LL_{\ttt}$ of (a) in Proposition \ref{prop:hermite}. (2) is just the above Lemma. Let us now prove (3).
	
	Let $(g_n-h_n\ttt,h_n)_{n\geq 0}$ be the sequence of Hermite vectors in $\LL_{\ttt}$ and $X_n(\ttt)=(p_n-q_n\ttt,q_n)$, $n\geq 0$, be the complete sequence of minimal vectors of $\LL_{\ttt}$. We can suppose that the $q_n$ and $h_n$ are $\geq 0$. By Lemma \ref{lem:HermiteMinimal}, the sequence $(h_n)_{n\geq 0}$ is a sub-sequence of the sequence $(q_n)_{n\geq 0}$. Therefore, there exists an increasing sequence $(n_k)_{k\geq 0}$ such that for all $k\geq 0$, $h_k=q_{n_k}$. By definition,  
	\[
	\{0\leq n<1+n_k: X_k(\ttt) \text{ is a Hermite vector}\}=
		\{n_0,\dots,n_k\},
	\]
	and by the above Lemma, for almost all $\ttt$,
	\[
	\lim_{k\fff\infty}\frac{k+1}{1+n_k}=\frac{\ln 3}{2\ln 2},
	\]
	so that by Levy's Theorem (\cite{Lev}),
	\begin{align*}
		\frac{1}{k+1}\ln h_k&=	\frac{1}{k+1}\ln q_{n_k}\\
		&=\frac{n_k}{k+1}\times\frac{1}{n_k}\ln q_{n_k}	\\
		&\longrightarrow \frac{2\ln 2}{\ln 3}\times \frac{\pi^2}{12 \ln 2}=\frac{\pi^2}{6 \ln 3}
	\end{align*}
	when $k$ goes to infinity.
\end{proof}

\end{document}